\documentclass[12pt]{article}
\usepackage[utf8]{inputenc}
\usepackage[T1]{fontenc}
\usepackage{amsmath,amsthm}
\usepackage{amssymb, enumitem}
\usepackage{cleveref}
\usepackage[english]{babel}
\usepackage{xcolor}

\newtheorem{lemma}{Lemma}
\newtheorem{thm}{Theorem}
\newtheorem{cor}{Corollary}

\newtheorem{ass}{Assumption}

\theoremstyle{definition}
\newtheorem{remark}{Remark}

	\title{Morse index and determinant of block Jacobi matrices via optimal control}

\author{Stefano Baranzini \footnote{Università degli studi di Torino,  Via Carlo Alberto 10, Torino, Italy } \and Ivan Beschastnyi \footnote{Univesidade de Aveiro, Campus Santiago, Aveiro, Portugal}}

\begin{document}
	\maketitle
%\begin{frontmatter}
	
	

	  %\ead{ibeschastnyi@ua.pt}
%     %\affiliation[aveiro]{organization={Univesidade de Aveiro},
%		            addressline={Campus Santiago},
%		            city={Aveiro},
%		             postcode={3810-193},
%		             %state={},
%		             country={Portugal}}

%		 	  \ead{sbaranzi@sissa.it}
% \affiliation[torino]{organization={Università degli studi di Torino},
%		             addressline={Via Carlo Alberto 10},
%		             city={Torino},
%		             postcode={10124},
%		             %state={},
%		             country={Italy}}

	\begin{abstract}
		 We describe the relation between block Jacobi matrices and minimization problems for discrete time optimal control problems. Using techniques developed for the continuous case, we provide new  algorithms to compute spectral invariants of block Jacobi matrices. Some examples and applications are presented. 		
	\end{abstract}

%	\begin{keyword}
%		%% keywords here, in the form: keyword \sep keyword
%		Jacobi matrix \sep Discrete systems \sep Morse index \sep Determinant \sep Toepliz matrix.
%		%% PACS codes here, in the form: \PACS code \sep code
%		
%		%% MSC codes here, in the form: \MSC code \sep code
%		%% or \MSC[2008] code \sep code (2000 is the default)
%		
%	\end{keyword}
%	
%\end{frontmatter}


    \section{Introduction}
    \label{sec:intro}
\subsection{Motivation}   
    \label{sec:motivation}
The goal of this paper is to explore an interesting connection between block Jacobi matrices and a class of discrete optimal control problems. This gives effective algorithms for computing the determinant, the negative inertia index and more generally the number of eigenvalues smaller than some threshold $\lambda^* \in \mathbb{R}$  of large block Jacobi matrices. 

Recall that a block Jacobi matrix $\mathcal{I}$ is a matrix of the form
    \begin{equation}
       \label{eq: jacobi_matrix_blocks}
       {\mathcal{I}} = \begin{pmatrix}
S_1 & R_1 & 0 & \dots & 0 & 0\\
R_1^t & S_2 & R_2 & \dots & 0& 0\\
0 & R_2^t & S_3 & \dots & 0& 0\\
\vdots & \vdots& \vdots & \ddots & \vdots & \vdots\\
0 & 0 & 0 & \dots & S_{N-1} & R_{N-1}\\
0 & 0 & 0 & \dots & R^t_{N-1} & S_N
        \end{pmatrix},
    \end{equation}
where $N\in \mathbb{N}$, $S_k$ are symmetric matrices of order $n$. Usually it is also assumed that $R_k$ are positive symmetric matrices of the same order. Here, we only require them to be non-degenerate, i.e. $R_k \in GL(n,\mathbb{R})$.

Jacobi matrices find applications in numerical analysis \cite{motiveation_numerical}, statistical physics \cite{ motivation_ising}, knot theory \cite{motivation_knots} and many other areas. Moreover, any symmetric matrix can be put in a tridiagonal form using the Householder method \cite{numanal_book}. Therefore, understanding their spectral properties is an extremely important topic. 

 In this article we establish a correspondence between Jacobi matrices and discrete linear quadratic regulator (LQR) problem. Then we show how the general machinery of optimal control theory allows us to compute the index and the determinant of such matrices. Similar formulas for determinants also appeared in \cite{det_old}, see also  \cite{SCHULZBALDES2012498,schulz2020space} for a geometric approach to the problem.
 
The LQR approach nevertheless has its benefits. In particular, as a small application of our method we give a novel proof of the generalized Euler identity
\begin{equation}
\label{eq:euler_agrachev}
\frac{a \sin b}{b\sinh a } = \prod_{j=1}^\infty\left(1-\frac{a^2+b^2}{a^2 + (\pi j)^2} \right)
\end{equation}
using only methods of linear algebra and basic analysis.

In order to state the main results, let us introduce the necessary notations. Consider the following discrete differential equation:
     \begin{equation}
     	\label{eq: difference eq}
     	x_{k+1}-x_k = A_k x_k +B_k u_k, \qquad k=0,\dots, N+1,
     \end{equation}
Where $A_k$ and $B_k$ are $n \times n$ matrix, $B_k$ are invertible and $x_k,u_k\in \mathbb{R}^n$. Variables $x_k$ are called \emph{state} variables while $u_k$ are the \emph{controls}. We denote by  $x = (x_k)$ and $u = (u_k)$ the vectors having $x_k$ and $u_k$ respectively as $k$-th component.  Note that we will often consider them as elements of the space of functions on a finite set of points with values in $\mathbb R^n$, clearly isomorphic to $\mathbb{R}^{n (N+2)}$  and $\mathbb{R}^{n (N+1)}$ correspondingly. 
On the space of solutions of \cref{eq: difference eq} with fixed initial condition $x_0$ we consider the following quadratic functional given by:
    \begin{equation}
        \label{eq: functional}
    	\mathcal{J}(u) = \frac{1}{2}\sum_{i=0}^N \left (\vert u_i\vert ^2- \langle Q_i x_i,x_i\rangle \right),
    \end{equation}
where $Q_i$ are some symmetric $n\times n$ matrices.

The problem of minimizing $\mathcal{J}$ over all possible discrete trajectories of~\eqref{eq: difference eq} is the classical linear quadratic regulator problem, which is one of the central problems in optimal control theory. Since we assume that $B_k$ are invertible, we can resolve the constraints~\eqref{eq: difference eq} with respect to $u_k$ and plug them into~\eqref{eq: functional}. This yields
     \begin{equation*}
    	\mathcal{J}(x) =  \frac{1}{2}\left (\sum_{i=0}^N \vert B_i^{-1}(x_{i+1}-(1+A_i)x_i)\vert ^2- \sum_{i=0}^{N} \langle Q_i x_i,x_i\rangle \right).
     \end{equation*} 
If we assume additionally that $x_0 = 0 = x_{N+1}$ and write down the matrix associated to the quadratic form above, we find that it is given by a block Jacobi matrix~\eqref{eq: jacobi_matrix_blocks}. Indeed, set 
$$
\Gamma_k^{-1} := (B_k^{-1})^tB_k^{-1}, 
$$ 
which is a symmetric non degenerate matrix. Then the functional $2\mathcal J$ is the represented by the symmetric matrix~\eqref{eq: jacobi_matrix_blocks} with:
   \begin{align}
   	S_k &= \Gamma_{k-1}^{-1}+(1+A_k)^t\Gamma_{k}^{-1}(1+A_k)-Q_k, \nonumber\\
   	 R_k &= -(1+A_k)^t \Gamma_k^{-1}, \label{eq:jacobi_transform}\\
   	 \,  &1 \le k \le N-1.\nonumber 
   \end{align}

We now see that to each LQR problem~\eqref{eq: difference eq}-\eqref{eq: functional} we can associate a Jacobi matrix. The converse is also true. However, there are several LQR problems which correspond to the same Jacobi matrix. We can get rid of this ambiguity by choosing fixed invertible $B_k$, $0\leq k \leq n$ and an arbitrary matrix $A_0$. Thus to every LQR problem~\eqref{eq: difference eq}-\eqref{eq: functional} one can associate a Jacobi matrix and vice versa. Fixing $B_k$, $0\leq k \leq n$ and $A_0$ beforehand makes this correspondence an affine bijection.

\subsection{Main results}
\label{sec:main results}
Through out the paper we will make the following assumption. It is equivalent to requiring $R_k \in GL(n,\mathbb{R})$ in \eqref{eq:jacobi_transform}.
\begin{ass}
	\label{assumption:invertibility}
	The matrices $B_k$ and $1+A_k$ in \eqref{eq: difference eq} are invertible for every $k$. 
\end{ass} 

Similarly to the case of continuous optimal control problems, one can characterize the extremal points of the functional~\eqref{eq: functional} under the constraints~\eqref{eq: difference eq} as solutions to a system of difference equations.
To be more precise define the following matrices for $0\leq i\leq j\leq N+1$ 
    \begin{equation*}
    	P_i^{j} : = \prod_{r = i}^{j-1}(1+A_r), \quad P_j^j = 1.
     \end{equation*}
and for $0 \leq i<j \leq N+1$ we take:
$$
P_j^i = (P_i^j)^{-1}. 
$$  
Next, we define $2n\times 2n$ matrices:
\begin{align*}
M_k = 
	\begin{pmatrix}
	(P^k_{k+1})^t  &0\\
	\Gamma_k(P^k_{k+1})^t & P^{k+1}_k
\end{pmatrix} \begin{pmatrix}
	1 & -Q_k \\ 0 &1
\end{pmatrix}
\end{align*}
and
$$
\Phi_{k} = \begin{pmatrix}
		\Phi^1_k & \Phi^2_k \\\Phi^3_k & \Phi^4_k
	\end{pmatrix} := \prod_{i=0}^{k-1} M_{i}.
$$
Matrices $\Phi_k$ correspond to the flow of the extremal equations that we will derive in Section~\ref{section:lagrange multiplier} using the Lagrange multiplier rule.
Now we can formulate the two main results of this paper.

\begin{thm}
\label{thm:index}
Let $\mathcal{J}$ be as in~\eqref{eq: functional} and consider constraints as in~\eqref{eq: difference eq}. The negative inertia index of $\mathcal{J}$ restricted to solutions of~\eqref{eq: difference eq} with $x_0=x_{N+1}=0$ is given by the following formula:
\begin{equation}
\label{eq:index_formula}
ind^- \mathcal{J} = \sum_{k=1}^{N}ind^-  \mathcal{Q}_k^\perp + \dim \ker \Phi^3_k, 
\end{equation}
where
$$
\mathcal{Q}_k^\perp = (\Phi_{k}^3)^t\left((P^{k+1}_{k})^{t}\Gamma_{k}^{-1}(P^{k+1}_{k})-Q_{k} \right)\Phi_{k}^3 + (\Phi_{k}^3)^t\Phi_{k}^1
$$
\end{thm}

Notice that we have reduced the problem of computing the negative inertia index of a $nN\times nN$ matrix to computing the inertia of  $N$ matrices of dimension $n\times n$, which  greatly reduces  the dimensionality of the problem. In the same spirit, we also prove a formula for the determinant.

\begin{thm}
\label{thm:det}
Let $\mathcal{J}$ be as in~\eqref{eq: functional} and consider constraints as in~\eqref{eq: difference eq}. The determinant of $\mathcal{J}$ restricted to solutions of~\eqref{eq: difference eq} with $x_0=x_{N+1}=0$ is given by the following formula:
\begin{equation}
	\label{eq:det_formula}	
	\det (\mathcal J) = 2^{-nN}\det(\Phi_3)\det(P_0^{N+1}) \prod_{k=0}^{N}\det\Gamma_k^{-1}.
\end{equation}
\end{thm}

	Let us rephrase the two statements in terms of Jacobi matrices of the form \eqref{eq: jacobi_matrix_blocks}. We can choose $A_0 =0$ and $\Gamma_k =1 $ for all $0\le k \le N$ and obtain the following relations from \cref{eq:jacobi_transform}  for $1\le k\le N-1$:
	\begin{equation*}
		\begin{cases}
			Q_k = 1-S_k+R_kR^t_k,\\
			P_k^{k+1} = -R_k^t,
		\end{cases}
	\end{equation*}
    \begin{equation*}
        M_k = \begin{pmatrix}
		-R^{-1}_k & R_k^{-1}(1-S_k+R_kR_k^t) \\ -R_k^{-1} &R_k^{-1}(1-S_k)
	\end{pmatrix}  .  
	\end{equation*}

\begin{cor}
	Index and determinant of the matrix $\mathcal{I}$ defined in \cref{eq: jacobi_matrix_blocks} can be computed as:
	\begin{align*}
		ind^- \mathcal{I} &= \sum_{k=1}^{N}ind^-  \mathcal{Q}_k^\perp + \dim \ker \Phi^3_k,\\
		\det (\mathcal I) &= (-1)^{nN} \det(\Phi_3)\prod_{k=1}^{N}\det R_k.
	\end{align*}
where $	\mathcal{Q}_k^\perp = (\Phi_{k}^3)^t\left(S_k-1\right)\Phi_{k}^3 + (\Phi_{k}^3)^t\Phi_{k}^1 $.
\end{cor}

\begin{remark}
	Clearly, applying formula \eqref{eq:index_formula} with $\tilde Q_k = Q_k+\lambda^*$ instead of $Q_k$ (or equivalently $\tilde S_k = S_k-\lambda^*$ instead of $S_k$), gives a formula to compute the  number of eigenvalues of $\mathcal{J}$ (respectively of $\mathcal{I}$) smaller than $\lambda^* \in \mathbb{R}$. For some oscillation results in this setting, see  \cite{Ammann2008RelativeOT}.
\end{remark}

The techniques used in our proofs are deeply connected to symplectic geometry. One can extend them to much more general settings \cite{beschastnyi_morse,morsegraph,hillBaranzini}. We have chosen to keep the exposition as simple as possible and minimize the use of symplectic language in this work. In this way, both formulas~\eqref{eq:index_formula} and~\eqref{eq:det_formula} can be applied directly.

The article has the following structure. In Section~\ref{section:lagrange multiplier} we derive the Hamiltonian system for the extremal curves using the end-point map and Lagrange multiplier rule. In Section~\ref{sec:index} we prove Theorem~\ref{thm:index} and in Section~\ref{sec:determinant}  Theorem~\ref{thm:det}. Finally, in Section~\ref{sec:applications}, we consider some examples and prove the generalized Euler identity~\eqref{eq:euler_agrachev}.

  \section{Eigenvalues of Jacobi matrices and the Lagrange multiplier rule}
  \label{section:lagrange multiplier}

In order to compute the index and the determinant of a Jacobi matrix $ {\mathcal{I}}$ in \eqref{eq: jacobi_matrix_blocks} we will use the flow of a certain system of difference equations. First, though, we need some notions from optimal control theory. 

The first object we introduce is the \textit{endpoint map}: 
$$
E^{k+1}: u\mapsto x_{k+1}(u), \qquad 0\leq k \leq N,
$$ 
This map takes a control $u$ and gives the solution of the differential equation ~\eqref{eq: difference eq} with $x_0 = 0$ at step ${k+1}$. We can write this map explicitly by iterating~\eqref{eq: difference eq}. 
    \begin{equation}
    	\label{eq:endpoint map}
    	E^{k+1}(u) = x_{k+1}(u_0,\dots,u_k)=  \sum_{j=1}^{k+1} P_j^{k+1} B_{j-1}u_{j-1}.
    \end{equation}
    In particular, from this formula, it is clear that the value of $E^{k+1}(u)$ depends only on  the first $k$ components of the control $u$. It makes sense, thus, to introduce the following filtration (i.e. \emph{flag} of subspaces) in the space of controls. Set $U : = \mathbb{R}^{n  N}$ the space of controls and define:
    \begin{equation}
    \label{def:Uk}
    	U_k := \{(u_0,u_1, \dots, u_k, 0,\dots,0) : u_j \in \mathbb{R}^n\}.
    \end{equation}
    Denote by $pr_l$ the orthogonal projection on $U_l$ with respect to the standard Euclidean scalar product. We have:
    \begin{equation*}
    	E_{x_0}^{k+1}(u) = E_{x_0}^{k+1}(pr_l u) \quad \forall \, l\ge k.
    \end{equation*}

Under \Cref{assumption:invertibility}, we can resolve the constraints $u\in \ker E^{k+1} \cap U_k$ quite easily. This will play an important role in all of the proofs. Indeed, consider the map 
$$
F_k:U_{k-1}\to \mathbb R^n,
$$
defined as
\begin{equation}
\label{eq:F_k_def}
F_k (u) := -B_{k}^{-1}\sum_{j=1}^{k}P^{k+1}_{j}B_{j-1} u_{j-1}.
\end{equation}
It is straightforward to check, using \cref{eq:endpoint map}, that:
\begin{equation}
\label{eq:resolution_using_F}
(u_0, \dots,u_k) \in \ker E^{k+1}\cap U_k \iff u_k = F_k(u).
\end{equation}
We can identify $F_k$ with a row block-matrix of size $n \times nk$. 

\bigskip

As we have seen in the previous section, a Jacobi matrix can be interpreted as the functional $\mathcal J$ in \cref{eq: functional} restricted to $\ker E^{N+1}$. We would like to find the eigenvalues of this restriction, i.e. of the corresponding Jacobi matrix ${\mathcal{I}}$. 

By definition $\tau\in \mathbb R$ is an eigenvalue of $\mathcal{J}$ if and only if the bilinear form $\mathcal J - \tau|u|^2$ is degenerate on $\ker E^{N+1}$. Since the form $\mathcal J(u) - \tau |u|^2$ is quadratic and $\ker E^{N+1}$ is a linear space, this is equivalent to finding critical points of $\mathcal J(u) - \tau |u|^2$ on $\ker E^{N+1}$. This is a typical task in constrained optimization and can be done via Lagrange multiplier rule. Notice that $\tau = \frac{1}{2}$ is not an eigenvalue of $\mathcal{J}$ if and only if all $Q_i$, $1\leq i \leq N$ are non degenerate.

Let us multiply $\mathcal{J}-\tau$ by $(1-2\tau)^{-1}$, operation that clearly preserves the kernel. After we introduce a new parameter $s = 1/(1-2\tau)$ we obtain a family of functionals
\begin{align*}
\mathcal J_s(u) :&= s\mathcal J(u) - \frac{s-1}{2}|u|^2 \\
&= \frac{1}{2}\left (\sum_{i=0}^N \vert u_i\vert ^2- s\sum_{i=1}^{N} \langle x_i,Q_i x_i\rangle \right).
\end{align*}

The following result is a discrete version of the Pontryagin maximum principle for LQR problems, which characterizes critical points of the constrained variational problem as solutions to a Hamiltonian system.

\begin{lemma}
	\label{lemma: jacobi equation}
The quadratic form $\mathcal J_s$ is degenerate on $\ker E^{N+1}$ if and only if the boundary value problem
\begin{equation}
\label{eq:jacobi_equation}
\begin{cases}
x_{k+1} = P^{k+1}_k x_k +B_k B_k^t \lambda_{k+1},\\
\lambda_{k} = (P^{k+1}_k)^t\lambda_{k+1} +s Q_k x_k,
\end{cases}
\end{equation}
with $x_0 = 0$, $x_{N+1} = 0$ has a non-trivial solution. Moreover the dimension of the space of solutions of \cref{eq:jacobi_equation} and of $\ker \mathcal J_s|_{\ker E^{N+1}}$ coincide. 
 
The former quadratic form is degenerate if and only if $(1-s)/s$ is an eigenvalue of the restriction to $\ker dE^{N+1}$ of:
 \begin{equation*}
\sum_{i=0}^N \vert v_i\vert^2 -  \sum_{j=1}^{N+1}( E^i(v))^t Q_i E^i(v).
 \end{equation*}
\end{lemma}

\begin{proof}
If $u$ is critical point of $\mathcal{J}_s$ restricted to $\ker E^{N+1}$, there exists $\lambda_{N+1}\in \mathbb R^n$ such that

\begin{equation}
\label{eq:lagragrange_rule}
	d_u\mathcal{J}^N = \lambda_{N+1}^t d_u E^{N+1}.
\end{equation}
Since $E^{N+1}$ is linear, we have $d_{u} E^{N+1} = E^{N+1}$. We will often use the latter notation in the formulas that follow.

In order to obtain a discrete analogue of a Hamiltonian system of PMP, we need to introduce for $1\leq k \leq N$ a family of functionals
$$
\mathcal J_s^k (u) := \frac{1}{2}\left (\sum_{i=0}^k \vert u_i\vert ^2- s\sum_{i=1}^{k} \langle x_i,Q_i x_i\rangle \right)
$$
and their differentials:
\begin{equation}
		d_u\mathcal{J}^k_s(v) = \sum_{i=0}^k  \langle u_i,v_i \rangle -  s\sum_{i=1}^{k}  x_i^t Q_i E^i(v).\label{eq:functional_differ}
\end{equation}
Functionals $\mathcal{J}^k_s$ differ from $\mathcal{J}_s$ only by the range of summation and clearly $\mathcal{J}^N_s=\mathcal{J}_s$. We look for $\lambda_k \in \mathbb{R}^n$, $1\leq k \leq N+1$ which satisfy:
\begin{equation}
\label{eq:restricted_lag_rule}
	d_{u}\mathcal{J}^k_s = \lambda_{k+1}^t d_uE^{k+1},
\end{equation}

Using the explicit formula~\eqref{eq:endpoint map} for the end-point map we find the recurrence relation 
\begin{equation}
\label{eq:end_p_map_differ}
	E^{k+1}(v) = P^{k+1}_k E^{k}(pr_{k-1}v) + B_k v_k
\end{equation}
for all $v\in U_k$.
From here, we can already derive the expression for the $k-th$ component of the control. Indeed, assume that $v_i=0$ for all $i< k$. In this case 
$$
E^{i}(pr_{i-1} v) = 0, \quad \forall i\le k.
$$
Hence, if we substite such a control $v$ in~\eqref{eq:restricted_lag_rule}, then using~\eqref{eq:functional_differ} and~\eqref{eq:end_p_map_differ} we obtain the control law
\begin{equation}
\label{eq:opt control}
u_k = B_k^t \lambda_{k+1},
\end{equation}
which we can plug in~\eqref{eq: difference eq}.

The next step is to obtain a discrete equation for $\lambda_k$. To do this we compare the multipliers $\lambda_k$ and $\lambda_{k+1}$ by subtracting~\eqref{eq:restricted_lag_rule} from the same formula with $k$ replaced by $k-1$. Using $v\in U_{k-1}\subset U_k$ as a test variation and formulas~\eqref{eq:functional_differ} and~\eqref{eq:end_p_map_differ} we find that:
\begin{align*}
&\lambda_{k+1}^t E^{k+1}(v) - \lambda_{k}^t E^{k}(v) = d_u\mathcal{J}^{k}_s(v)- d_u\mathcal{J}^{k-1}_s(v)  \\
 &\iff \left(\lambda^t_{k+1}P^{k+1}_k -\lambda^t_k + s x_k^t Q_k \right) E^{k}(v) = 0. 
\end{align*}
Since by our assumption $B_k$ are invertible for all $0\leq k \leq N$, the end-point map $E^{k}$ is surjective, as can be seen from the explicit expression~\eqref{eq:endpoint map}. Thus the term in the bracket must vanish. 

Collecting everything gives \cref{eq:jacobi_equation}. The boundary conditions for this system come form the fact that we have to add the equation $E^{N+1}(u) =0 = x_{N+1}$ and that we are assuming $x_0 =0$.  

Now, notice that the initial covector $\lambda_0$ of the lift uniquely determines the whole trajectory and the control. This means that the correspondence $\lambda_0 \mapsto u$ is a bijection between the space of solutions of the boundary value problem and the kernel of $\mathcal{J}_s$.
\end{proof}

Under \Cref{assumption:invertibility}, we can rewrite system~\eqref{eq:jacobi_equation} as a forward equation. To do so, recall that $\Gamma_k = B_k B_k^t $, multiply the second equation by $(P^k_{k+1})^t$ and plug in the new expression for $\lambda_{k+1}$ into the first equation. This gives
\begin{equation}
	\label{eq: systemPMP}
\begin{cases}
x_{k+1} = (P^{k+1}_k -s \Gamma_k (P^k_{k+1})^t Q_k)x_k +\Gamma_k(P^k_{k+1})^t \lambda_{k},\\
\lambda_{k+1} = (P^k_{k+1})^t\lambda_{k}-s (P^k_{k+1})^t Q_k x_k.
\end{cases}
\end{equation}
We thus have
$$
\begin{pmatrix}
\lambda_{k+1}\\
x_{k+1}
\end{pmatrix} = M_k(s) 
\begin{pmatrix}
\lambda_{k}\\
x_{k}
\end{pmatrix},
$$
where
\begin{align*}
M_k(s) = 
	\begin{pmatrix}
	(P^k_{k+1})^t  &0\\
	\Gamma_k(P^k_{k+1})^t & P^{k+1}_k
\end{pmatrix} \begin{pmatrix}
	1 & -sQ_k \\ 0 &1
\end{pmatrix}.
\end{align*}
Both matrices in the product are symplectic, which makes $M_k(s)$ symplectic as well.

Define the flow up to the point $k$ as
$$
\Phi_{k}(s) = \begin{pmatrix}
		\Phi^1_k(s) & \Phi^2_k(s) \\\Phi^3_k(s) & \Phi^4_k(s)
	\end{pmatrix} := \prod_{i=0}^{k-1} M_{i}(s).
$$
When $k=N+1$, we just write $\Phi(s) = \Phi_{N+1}(s)$. Also in accordance with the notations given in the introduction, we write for $s=1$:
$$
\mathcal J^k := \mathcal J^k_1, \qquad \Phi_k := \Phi_k(1), \qquad M_k:= M_k(1). 
$$

\begin{remark}
We can reformulate Lemma~\ref{lemma: jacobi equation} and the boundary value problem~\eqref{eq:jacobi_equation} in terms of $\Phi_k(s)$. We are looking for solutions of~\eqref{eq:jacobi_equation} with boundary values $x_0 = x_{k+1} =0$. Writing as above $\Phi_{k+1}(s)$ as a block matrix shows that the existence of a solution is equivalent to the vanishing of the determinant of the block $\Phi^3_{k+1}$:
$$
\ker \mathcal J^k_s|_{\ker E^{k+1}} \neq 0 \iff \det \Phi^3_{k+1}(s) = 0.
$$
\end{remark}

%{\color{red}Non mi ricordo dove sta}

%System~\eqref{eq: difference eq} provides an isomorphism between the space of controls $u \in \mathbb{R}^{n(N+1)}$ and the space $x \in \mathbb{R}^{n(N+2)}$ with $x_0 =0$. 
%We can pass from controls variables to state variables via the following change of coordinates:
%\begin{equation}
%	\label{eq:change_trajectory_to_control}
%	u_k = B_k^{-1}(x_{k+1}-P_k^{k+1}x_k)
%\end{equation} 
%for $k = 0,\dots, N$.

\section{A recursive formula for the index}
\label{sec:index}

In this section we prove Theorem~\ref{thm:index}. The tools we will employ are essentially taken from linear algebra.  The same kind of ideas has been applied in \cite{morsegraph} to reformulate problems of intersection theory for paths in the Lagrange Grassmannian in terms of finite dimensional linear algebra. The classical approach to second order minimality conditions rephrases the problem of counting the number of negative eigenvalues of the Second Variation (the counter part of our quadratic form $\mathcal{J}$) as counting the intersection number of a curve in a finite-dimensional manifolds, known as the Lagrangian Grassmannian, with an appropriate submanifold. For further references on the continuum approach  see for instance \cite{DuistermaatIntersectionIndex,LongIndex,agrachev_quadratic_paper} and reference therein. 

In contrast to the continuous case, here we work only with an ordered sets of points constructed from the solutions of \cref{eq:jacobi_equation}. Still, all the information about the index can be recovered from them.

We will use the following lemma:
\begin{lemma}
	\label{lem:orthgonal_index}
	Suppose that a quadratic form $\mathcal Q$ is defined on finite dimensional vector space $X$ and a subspace $V \subseteq X$ is given. Denote by $V^{\perp_{\mathcal Q}}$ the $\mathcal Q-$orthogonal subspace to $V$ i.e.:
	\begin{equation*}
		V^{\perp_{\mathcal Q}}:= \{u \in X : \mathcal Q(u,v) =0, \forall v \in V\}.
	\end{equation*}  
Denote by $ind^- Q$ the number of negative eigenvalues of $Q$, then:
\begin{align*}
	ind^-\mathcal Q &= ind^-\mathcal Q\vert_V+ind^-\mathcal Q\vert_{V^{\perp_{\mathcal Q}}}\\&+\dim\big((V\cap V^{\perp_{\mathcal{Q}}})/(V\cap V^{\perp_{\mathcal Q}}\cap \ker \mathcal Q)\big).
\end{align*}
\end{lemma}

We will apply iteratively this formula to the subspaces $V_k := U_k \cap \ker E^{k+1}$. The subspaces $U_k$ were defined in~\eqref{def:Uk} and provide a filtration of our space of variations. One has a natural filtration $V_{k-1}\subset V_{k}$ as well. Indeed, if $u\in V_{k-1}$, then $x_{k}=E^k(u)=0$. Hence, if we extend the control $u$ to $V_k$ by taking $u_k = 0$, we get from~\eqref{eq:endpoint map}
$$
x_{k+1}=E^{k+1}(u) = P^{k+1}_k E^k(u) = 0.
$$ 
Notice also that for any $k$, the subspace $V_{k-1}$ has codimension $n$ in $V_{k}$. 

We denote for brevity 
$$
\mathcal Q_k:= \mathcal{J}^k|_{\ker E^{k+1}}, \quad V_k^{\perp}:=V_k^{\perp_{\mathcal Q_{k+1}}}, \quad \mathcal Q_k^\perp := \mathcal Q_k|_{V_{k-1}^\perp}.
$$
 Let us apply Lemma~\ref{lem:orthgonal_index} with $\mathcal{Q} = \mathcal{Q}_{k+1}$ and $V = V_{k}\subseteq V_{k+1} =X$. This gives:
\begin{align*}
		ind^-\mathcal Q_{k+1} &= ind^-\mathcal Q_{k}+ind^-\mathcal Q_{k+1}^\perp \\&+\dim\big((V_k\cap V_k^{\perp})/(V_k\cap V_k^{\perp}\cap \ker \mathcal Q_{k+1})\big).
\end{align*}

Thus, iteration of this formula gives:
\begin{align*}
	ind^- \mathcal{J}|_{\ker E^{N+1}} &= \sum_{k=0}^{N-1}ind^-\mathcal Q_{k+1}^\perp\\& +\dim\big((V_k\cap V_k^{\perp})/(V_k\cap V_k^{\perp}\cap \ker \mathcal  Q_{k+1})\big).
\end{align*}

What is left to do is to express every term using the discrete Hamiltonian system given in \cref{eq:jacobi_equation}.

First of all, let us describe the subspaces $V_k^\perp \cap V_k$ and $V_k^\perp \cap V_k\cap \ker \mathcal Q_{k+1}$. 
\begin{lemma}
 Let $V_k$  be defined as above and let $\Pi$ be the vertical subspace
$$
\Pi = \{(\lambda,x)\in \mathbb R^{2n}\,:\, x=0 \}.
$$
Then we have the following identifications:
\begin{equation*}
\begin{split}
   V_k^\perp \cap V_k &= \ker(\mathcal Q_{k}) = \Phi_{k+1}(\Pi)\cap \Pi;\\ 
    V_k^\perp \cap V_k \cap \ker(\mathcal Q_{k+1}) &= \Phi_{k+1}(\Pi)\cap M_{k+1}^{-1}(\Pi) \cap \Pi.
   \end{split}
 \end{equation*}
\end{lemma}

\begin{proof}
A direct consequence of the definitions of $V_k$ and $V_k^\perp$ is that $V_k^\perp \cap V_k = \ker(\mathcal Q_{k})$. \Cref{lemma: jacobi equation} identifies the elements of the kernel with solutions of the equation~\eqref{eq:jacobi_equation} with $x_0 = x_{k+1}=0$. In terms of the corresponding flow $\Phi_{k+1}$ we get:
$$
\ker(\mathcal Q_{k}) = \Phi_{k+1}(\Pi)\cap \Pi. 
$$

Similarly we have $\ker(\mathcal Q_{k+1}) = \Phi_{k+2}(\Pi)\cap \Pi$. Applying $M_{k+1}^{-1}$ gives 
$$
\ker(\mathcal Q_{k+1}) = \Phi_{k+1}(\Pi)\cap M_{k+1}^{-1}(\Pi).
$$
Combining this with the previous formula proves the second isomorphism in the statement.
\end{proof}

Notice that the previous lemma holds for general LQR. In this particular case, $M^{-1}_{k+1}(\Pi)$ is transversal to the fibre. Indeed, from the explicit expression of $M_{k+1}$ we have that $M_{k+1}(\Pi) \cap \Pi \neq 0$ if and only if $\ker \Gamma_{k+1} (P^{k+1}_{k+2})^t \neq 0$, which is never the case under \Cref{assumption:invertibility}. Thus
\begin{equation*}
\begin{split}
   V_k^\perp \cap V_k = \ker(\mathcal Q_{k}) &= \Phi_{k+1}(\Pi)\cap \Pi;\\ 
    V_k^\perp \cap V_k \cap \ker(\mathcal Q_{k+1}) &= 0.
   \end{split}
 \end{equation*} 

The next step is to compute $ind^- \mathcal Q_k^\perp$. To do so consider the standard symplectic form on $\mathbb R^{2n}$:
$$
\sigma((\lambda_1,x_1),(\lambda_2,x_2)) = \lambda_1^t x_2 - \lambda_2^t x_1=
\begin{pmatrix}
\lambda_1^t & x_1^t
\end{pmatrix}
\begin{pmatrix}
0 & I\\
-I & 0
\end{pmatrix}
\begin{pmatrix}
\lambda_2\\
x_2
\end{pmatrix}.
$$
A plane $\Lambda\subset \mathbb R^{2n}$ is called \textit{Lagrangian} if $\sigma|_{\Lambda} = 0$ and $\dim \Lambda = n$, i.e., it is a maximally isotropic space.

The \textit{Maslov form} $m(\Lambda_1,\Lambda_2,\Lambda_3)$ of a triple of Lagrangian spaces $\Lambda_1,\Lambda_2,\Lambda_3$ is a symmetric quadratic form on $(\Lambda_1 + \Lambda_3) \cap \Lambda_2$ defined as follows. Let $p_2 = p_1 + p_3$ where $p_i\in \Lambda_i$. Then
$$
m(\Lambda_1,\Lambda_2,\Lambda_3)(p_2) = \sigma(p_1,p_3). 
$$

\begin{lemma}
 The negative index of $\mathcal Q_{k}$ restricted to $V_{k-1}^\perp$ coincides with the negative index of $m(\Phi_{k}(\Pi),M_{k}^{-1}(\Pi),\Pi)$.
 \end{lemma}
	\begin{proof}

An element $u$ belongs to $V_k$ if and only if it is of the form:
\begin{equation*}
	u = (u_0,\dots,u_{k-1}, F_k(u)), 
\end{equation*}
where $F_k$ is defined in~\eqref{eq:F_k_def}. Hence we can identify $V_k$ with $\mathbb R^{nk}$.

%It is somewhat simpler notationally to drop all the zero after the $(k+1)-$th entry and thus work on the image of $pr_{k+1}$. We will identify $V_{k+1}$ and $V_k$ with the following spaces:
%\begin{equation*}
%V_k = \{(u,F_k(u),0):u \in \mathbb{R}^{nk}\}, \quad  V_{k+1} = \{(u,u_k,F_{k+1}(u,u_k)):u \in \mathbb{R}^{nk}\}.
%\end{equation*} 
It follows that an element $u$ belongs to $V_{k-1}^\perp$ if and only if for all $v\in V_{k-1}\simeq \mathbb R^{n(k-1)}$:
\begin{equation*}
	\begin{aligned}
&0=\mathcal Q_{k}(u,v) = \frac{1}{2}\left(\sum_{i=0}^{k} \langle u_i,v_i \rangle - \sum_{i=1}^{k} \langle Q_{i}E^{i}u,E^{i}v\rangle\right) \\
&= \frac{1}{2} \Bigg(\sum_{i=0}^{k-2}\langle u_i,v_i \rangle - \langle (E^{i+1})^t Q_{i+1} E^{i+1}u, v\rangle_{L^2} +  \langle u_{k-1}, F_{k-1} (v)\rangle\Bigg)\\
 &= \frac{1}{2} \langle  u-\sum_{i=0}^{k-2}(E^{i+1})^tQ_{i+1} E^{i+1}u+ F_{k-1}^t(u_{k-1}),v\rangle_{L^2}
\end{aligned}
\end{equation*}
where the $L^2$-scalar product is defined on functions on a finite set with values in $\mathbb R^n$. Here we used that $E^k(v) =0$ since $v \in V_{k-1}$.

We have thus shown that if $u\in V_{k-1}^\perp$ then
\begin{equation*} \langle  u-\sum_{i=0}^{k-2}(E^{i+1})^tQ_{i+1} E^{i+1}u,v\rangle_{L^2}  = - \langle F_{k-1}^t(u_{k-1}),v\rangle_{L^2}, 
\end{equation*}
for all functions $v:\{0,\dots,k-1\}\to \mathbb R^n$. 
	
Let us now simplify the form $\mathcal Q_{k}^\perp$ using this expression. Take an element $u \in V_{k-1}^\perp$. Then

     \begin{equation*}
     	\begin{split}
     		&\mathcal Q_k(u)
     		= \frac{1}{2}\Bigg(\left(\sum_{i=0}^{k-2}  \langle u_i, u_i \rangle - \langle u, (E^{i+1})^t Q_{i+1} E^{i+1}(u)\rangle \right) \\& + \langle u_{k-1},u_{k-1}\rangle+ \langle F_{k}(u),F_{k}(u)\rangle-\langle E^{k}(u),Q_{k}E^{k}(u)\rangle \Bigg)\\
     		&= \frac{1}{2} (-\langle F_{k-1}(u),u_{k-1} \rangle + \langle u_{k-1},u_{k-1}\rangle  + \langle F_{k}(u),F_{k}(u)\rangle\\&-\langle x_{k},Q_{k}x_{k}\rangle ).
     	\end{split}
     \end{equation*}
 Where we used the characterization of $V_{k-1}^\perp$ given in the previous lemma and substituted $E^{k}(u)$ with $x_{k}$. We now rewrite each term in symplectic form. In particular recall the following relations, which are consequences of~\eqref{eq:opt control} and~\eqref{eq: difference eq}
 \begin{equation*}
 	\begin{aligned}
u_k &= B_k^t\lambda_{k+1}, \\
x_{k}  &= P_{k-1}^{k}x_{k-1}+ \Gamma_{k-1} \lambda_{k},\\
 B_{k-1}^t\lambda_{k} &= F_{k-1}(u) = -B_{k-1}^{-1} P_{k-1}^{k}x_{k-1}.
 	\end{aligned}
 \end{equation*}
 This implies that:
\begin{align*}
\langle u_{k-1},u_{k-1}\rangle & = \langle  \lambda_{k},\Gamma_{k-1}\lambda_{k}\rangle,\\
\langle F_{k-1}(u),u_{k-1}\rangle &= -\langle \lambda_{k},P_{k-1}^k x_{k-1}\rangle.
\end{align*}

Similarly, $F_k(u)$ and $x_{k}$ can be expressed in terms of $\lambda_{k+1}$.
\begin{equation*}
- P_k^{k+1}x_k = B_k F_k(u) = \Gamma_k \lambda_{k+1}.
\end{equation*} We thus get the following expression for $u\in V_{k-1}^\perp$ after substituting the previous relations and grouping separately terms with $\lambda_k$ and $\lambda_{k+1}$:
\begin{equation*}
	\begin{split}
	\mathcal Q_k(u) &= \frac{1}{2} ( \langle \lambda_{k},P_{k-1}^k x_{k-1}+\Gamma_{k-1} \lambda_k\rangle  \\& \quad + \langle \lambda_{k+1},\Gamma_{k}\left( \lambda_{k+1} +(P_{k+1}^k)^tQ_{k}x_{k} \right) \rangle )\\
&=\frac{1}{2} \left( \langle \lambda_k,x_k\rangle +\langle \lambda_{k+1},\Gamma_k\left( \lambda_{k+1} +(P_{k+1}^k)^tQ_{k}x_{k} \right) \rangle \right).
\end{split}
\end{equation*} 
On the other hand let us consider the Maslov form $m(\Phi_{k}(\Pi),M^{-1}_{k}(\Pi),\Pi)$ and write down its explicit expression. We have on $M^{-1}_{k}(\Pi) \cap (\Phi_{k}(\Pi) +\Pi )$.
\begin{align}
\label{eq:maslov_spaces}
	\nonumber M_{k}^{-1}\begin{pmatrix}
		\lambda_{k+1}\\0
	\end{pmatrix} &= \begin{pmatrix}
		\left((P_{k}^{k+1})^t-Q_{k}P_{k+1}^{k}\Gamma_{k}\right) \lambda_{k+1}\\
		-P_{k+1}^{k}\Gamma_{k}\lambda_{k+1}
\end{pmatrix} \\  &= \begin{pmatrix}
\lambda_{k}+\nu\\ x_{k}
\end{pmatrix} \in \Pi+\Phi_{k}(\Pi).
\end{align}
Thus the value of $m(\lambda_{k+1}) = \sigma ((\lambda_k,x_k),(\nu,0)) =  -\langle \nu,x_{k}\rangle$ is determined by:
\begin{equation}
	\label{eq:maslovFormLemma}
	\begin{split}
	m(\lambda_{k+1}) &= -\langle	\left((P_{k}^{k+1})^t-Q_{k}P_{k+1}^{k}\Gamma_{k}\right) \lambda_{k+1}-\lambda_{k},x_{k}\rangle \\	
	&=\langle  \lambda_{k+1},\Gamma_{k}(\lambda_{k+1}+(P_{k+1}^{k})^tQ_{k}x_{k})\rangle + \langle \lambda_{k},x_{k}\rangle\\
	&= 2\mathcal{Q}_k(u).
\end{split}
\end{equation}
Here, in the second equality we substituted $\nu$ from~\eqref{eq:maslov_spaces} and in the third equality $x_k$ from the same equation. It follows that negative inertia index of $\mathcal Q_k^\perp$ is the same as the negative inertia index of $m(\Phi_{k}(\Pi),M^{-1}_{k}(\Pi),\Pi)$.
	\end{proof}

To arrive at formula~\eqref{eq:index_formula} we need to write down explicitly the kernel of the Maslov form in terms of the initial covector $\lambda_0$. We have
\begin{equation*}
	\begin{pmatrix}
		\lambda_{k} \\ x_{k}
	\end{pmatrix} = \begin{pmatrix}
	\Phi_{k}^1 \lambda_0\\ \Phi_{k}^3\lambda_0
\end{pmatrix}, \quad \Phi_{k} = \begin{pmatrix}
\Phi_{k}^1 &\Phi_{k}^2 \\ \Phi_{k}^3 &\Phi_{k}^4
\end{pmatrix}.
\end{equation*}
Moreover, $\lambda_{k+1}$ is directly determined by $x_{k}$, see \cref{eq:maslov_spaces}, and thus by $\lambda_0$ as:\begin{equation*}
	\lambda_{k+1} = -\Gamma_{k}^{-1}P^{k+1}_k\Phi_{k}^3\lambda_0.
\end{equation*}

After inserting inside the explicit expression of the Maslov form given in \eqref{eq:maslovFormLemma} we find that the following quadratic form is the one to consider:
\begin{equation*}
\langle \lambda_0,\Big((\Phi_{k}^3)^t\big((P^{k+1}_{k})^{t}\Gamma_{k}^{-1}(P^{k+1}_{k})-Q_{k} \big)\Phi_{k}^3 + (\Phi_{k}^3)^t\Phi_{k}^1\Big)\lambda_0 \rangle.
\end{equation*}

Now each term in the formula for $ind^- \mathcal J|_{\ker E^{N+1}}$ is identified and collecting all the pieces together finishes the proof of Theorem~\ref{thm:index}.

\section{Proof of the determinant formula}
\label{sec:determinant}

In this section we prove Theorem~\ref{thm:det}. 
Consider the quadratic form $2\mathcal J$. As discussed in the introduction, we can pass from control variables $u$ to state variables $x$. This gives the Jacobi matrix $ {\mathcal{I}}$ once we impose the boundary conditions $x_0 = x_{N+1} = 0$.  This matrix decomposes as the sum of two terms. A diagonal matrix $Q$ and a tridiagonal one $T = (T_{i,j})_{i,j} $ given by:
\begin{equation*}
	\begin{aligned}
		{\mathcal I}_s = T - sQ, \quad 
		Q = \mathrm{diag}(Q_i), \, i = 1,\dots, N, \\ T_i = \Gamma_{i-1}^{-1}+(P_i^{i+1})^t\Gamma^{-1}_{i}P_i^{i+1}, \quad  T_{i,i+1} = -(P_i^{i+1})^t \Gamma^{-1}_{i}.
	\end{aligned}		
\end{equation*}
Notice,  moreover, that the matrix $T$ is positive definite, since it is a multiple of the quadratic form $\sum_{i=0}^{N} \vert u_i\vert ^2$. For this reason we can define its square root $T^{1/2}$ which is again positive definite and symmetric. 
We will use this observation to relate the determinant of $\Phi_3(s)$ to the characteristic polynomial of a suitable $nN$ matrix.

%\begin{lemma}
%\label{lemma:char_polynomial}
%    Let $W$ be $T^{1/2} \ker Q$ and denote by $\hat T - s \hat Q$ the restriction of $T-sQ$ to $W^\perp$.
%	The function $\det(\Phi_3(s))$ is a polynomial of degree $nN-\sum_{k=1}^N \dim(\ker(Q_k))$ and in particular it is a multiple of $\det(\hat T-s\hat Q)$ (and thus of $\det(T- sQ)$).
%\end{lemma}

\begin{lemma}
\label{lemma:char_polynomial}
   	The function $\det(\Phi_3(s))$ is a polynomial of degree $nN-\sum_{k=1}^N \dim(\ker(Q_k))$ and in particular it is a multiple of $\det(T-s Q)$.
\end{lemma}

	\begin{proof}
		The first thing we have to notice is that the entries of $\Phi(s)$ are polynomials of degree at most $N$ since $\Phi(s)$ is the product of $N$ affine in $s$ matrices. This implies that $\det(\Phi_3(s))$ has at most degree $nN$. 
		
		This is exactly the dimension of the space of controls. We know by \Cref{lemma: jacobi equation} that this polynomial has at least $d: = n N-\sum_i\dim(\ker Q_i)$ roots. This is because our perturbation $\mathcal{J}_s$ sees every eigenvalue of $\mathcal{J}$ but $1$, whose eigenspace corresponds to the kernel of $Q$. In particular when the $Q_i$ are all invertible, $\det(\Phi_3(s))$ has degree $nN$, vanishes on the same set as $\det(T-sQ)$ and thus they are multiples. In general we have:
		\begin{equation*}
			\begin{split}
				\det(T-sQ) &= \det(T)\det(1-sT^{-1/2}Q T^{-1/2})  \\ &= s^{nN} \det(T)\det(s^{-1}-T^{-1/2}Q T^{-1/2}) \\
				&= \lambda^{-nN}\det(T)\det(\lambda-T^{-1/2}QT^{-1/2})\vert_{\lambda =s^{-1}}.
			\end{split}
		\end{equation*}
	
	   Order the non zero eigenvalues of $T^{-1/2}Q T^{-1/2}$ as $\{\lambda_i\}_{i=1}^d$. Since  $\det(\lambda- T^{-1/2}QT^{-1/2})$  has $d$ roots different from zero and a root of order  $\sum_i\dim(\ker(Q_i))$ at $\lambda =0$ we can rewrite the formula as:
	    \begin{equation*}
	    	\det(T-sQ) = \prod_{i = 1}^{d} (1-s \lambda_i) \det(T) .
	    \end{equation*}
	     Thus we have that $\det(\Phi_3(s))$ and $\det(T-sQ)$ are both polynomials of degree $d$ vanishing on the same set and thus they must be multiples.
\end{proof}

	By Lemma~\ref{lemma:char_polynomial} there exists a constant $C \ne 0$ such that
		$$
		\det (T-sQ) = C\det \Phi_3(s).		
		$$
The constant can be computed by evaluating both sides at $s=0$. Using induction we can prove that
\begin{equation}
	\label{eq:value_Phi(0)}
\Phi(0) = 
\begin{pmatrix}
(P^0_{N+1})^T & 0\\
\sum_{i=0}^{N}P^{N+1}_{i+1}\Gamma_i (P^0_{i+1})^T & P^{N+1}_0
\end{pmatrix}.
\end{equation}
Hence
	    \begin{equation*}
	    	\begin{split}
    	\det(\Phi_3(s))\vert_{s=0} &= \det\left(\sum_{i=0}^{N}P^{N+1}_{i+1}\Gamma_i (P^0_{i+1})^T\right)\\
    	&=\det\left(\sum_{i=0}^{N}P^{N+1}_{i+1}\Gamma_i (P_{i+1}^{N+1})^T\right)\det(P^0_{N+1}).
    \end{split}
    \end{equation*}	
Now it only remains to find $\det T$. Recall that $T$ corresponds to the quadratic form $\frac{1}{2} \sum_{k=0}^{N} \vert u_k\vert^2$.

Given a trajectory of~\eqref{eq: difference eq} with $x_0=x_{N+1}=0$ we can always reconstruct the control. This gives a bijection between $(u_0,\dots,u_{N-1})$ and $(x_1,\dots x_N)$ which can be written using a block-triangular matrix
$$
L = \begin{pmatrix}
B_0^{-1} & 0  & \dots & 0 & 0 \\
-B_1^{-1}P^2_1 & B_1^{-1} & \dots & 0 & 0\\
\vdots & \vdots  & \ddots & \vdots & \vdots\\
0 & 0 & \dots & -B_{N-1}^{-1}P^{N}_{N-1} & B_{N-1}^{-1}
\end{pmatrix}
$$

Now we wish to use $F_N$ (see \cref{eq:F_k_def}) to pull back to $U_{N-1}$ the quadratic form:
$$
\sum_{i=0}^{N}|u_i|^2
$$ 
seen as a quadratic form on $\ker E^{N+1}$. To do so we simply need to resolve for $u_N$ using~\eqref{eq:resolution_using_F} and plug it inside the form above. This gives the form $q(v) = \langle (1+F_N^*F_N)v,v\rangle$ on $U_{N-1}$.  

We can now change coordinates and work with the $x$ variable. One can see that $T$ can be written as a symmetric matrix on $\mathbb R^{nN}\times \mathbb R^{nN}$ of the form:
$$
T = L^T(1+ F_N^*F_N)L.
$$
So if we want to compute its determinant, we need to find determinant of $L$ and of $1+ F_N^* F_N$. We clearly have
$$
\det L = \prod_{i=0}^{N-1} B_i^{-1}.
$$
So it only remains to find $\det (1+ F_N^* F_N)$.

Notice that, since $F_N$ is a surjection, it has a large kernel of dimension $N(n-1)$. If $v\in \ker F_N$, then $v$ is an eigenvector of $q$ with eigenvalue $1$. So the determinant can be computed multiplying the remaining $n$ eigenvalues. We claim that they are all equal to the eigenvalues of $1+F_NF_N^*$. Indeed, if $v\notin \ker F_N$ is an eigenvalue of $q$, then
$$
(1+ F_N^*F_N)v = \lambda v.
$$
Applying $F$ to both sides gives
$$
(1+ F_N F_N^* )F_Nv = \lambda F_Nv. 
$$
Hence we obtain that:
\begin{equation*}
\begin{split}
\det T &= \prod_{i=0}^{N-1}\det(\Gamma_i)^{-1} \\& \times\det \left(1+\sum_{j=0}^{N-1} B_N^{-1}P^{N+1}_{j+1} \Gamma_j^{-1}(P^{N+1}_{j+1})^t (B_N^{-1})^t \right)  \\
&=\prod_{i=0}^{N}\det(\Gamma_i)^{-1}\det \left(\Gamma_N+\sum_{j=0}^{N-1} P^{N+1}_{j+1} \Gamma_j^{-1}(P^{N+1}_{j+1})^t \right). 
\end{split}
\end{equation*}

This implies that $C = \det(P_0^{N+1}) \prod_{k=0}^N\Gamma_k^{-1}$ and thus:
\begin{equation}
	\det(\Phi_3(s))\det(P_0^{N+1}) \prod_{k=0}^{N}\det\Gamma_k^{-1} = \det(T-sQ),
\end{equation}
which gives formula~\eqref{eq:det_formula} when $s=1$.

\section{Some examples and applications}
\label{sec:applications}
\subsection{Tridiagonal Toeplitz Matrices}
\label{sec:toeplitz matrices}
Let us compute the determinant of a symmetric tridiagonal Toeplitz matrix
\begin{equation}
       \label{eq: toeplitz_matrix}
      T_N(s,r) = \begin{pmatrix}
s & r & 0 & \dots & 0 & 0\\
r & s & r & \dots & 0& 0\\
0 & r & s & \dots & 0& 0\\
\vdots & \vdots& \vdots & \ddots & \vdots & \vdots\\
0 & 0 & 0 & \dots & s & r\\
0 & 0 & 0 & \dots & r & s
        \end{pmatrix},
    \end{equation}
where $s,r \in \mathbb{R}$ and $r\neq 0$. 

Recall that eigenvalues of $T_N(s,r)$ are explicitly known \cite{tridiagonal_eigen} and are of the form
$$
\lambda_j = s+ 2|r|\cos \frac{\pi j}{N+1}, \qquad 1 \leq j \leq N. 
$$
So we already a have a way of computing the determinant of the matrix as the product of all eigenvalues:
\begin{equation}
\label{eq:det_toeplitz_eigenvalues}
\det T_N(s,r)=\prod_{j=1}^N \left( s+ 2|r|\cos \frac{\pi j}{N+1}\right).
\end{equation}
Let us compute the same determinant using the control theoretic approach.

Using formulas~\eqref{eq:jacobi_transform} we reformulate the problem of computing $\det T_N(r,s)$ as a LQR problem. We choose for $1\leq i \leq N$:
$$
B_0 = B_i = 1, \quad A_0 = 0, \quad A_i = -(1+r), \quad Q_i= 1+r^2-s. 
$$
In particular, this immediately implies that
$$
\Gamma_i =1, \qquad P^{N+1}_0 = (-r)^N.
$$
After comparing with~\eqref{eq:det_formula}, we see that only $\Phi_3$ needs to be computed. In this particular case
$$
\Phi = 
\begin{pmatrix}
-r^{-1} & r^{-1}(1+r^2 - s) \\
-r^{-1} & r^{-1}(1 - s) 
\end{pmatrix}^N
\begin{pmatrix}
1 & 0\\
1 & 1
\end{pmatrix}.
$$
Denote 
$$
M=\begin{pmatrix}
-r^{-1} & r^{-1}(1+r^2 - s) \\
-r^{-1} & r^{-1}(1 - s) 
\end{pmatrix}.
$$
The eigenvalues of this matrix are given by
\begin{equation}
\label{eq:eigenvalues_toeplitz}
\mu_{\pm}=\frac{-s\pm \sqrt{s^2-4r^2}}{2r}.
\end{equation}
Assume that $\mu_+ \neq \mu_-$. Then we can diagonalize $M$ using the matrix
$$
S= \begin{pmatrix}
1+r\mu_- & 1+r\mu_+\\
1 & 1
\end{pmatrix},
$$
as can be verified by a direct computation. Thus
$$
\Phi = 
S\begin{pmatrix}
\mu_+^N & 0 \\
0 & \mu_-^N 
\end{pmatrix} S^{-1}
\begin{pmatrix}
1 & 0\\
1 & 1
\end{pmatrix}.
$$
After all of the matrices are multiplied, one can take the element which lies under the diagonal and use formula~\eqref{eq:det_formula}. This way, after some simplification, we obtain an expression which can be compactly written in terms of the two eigenvalues
\begin{equation}
\label{eq:det_Toeplitz_control}
\det T_N(s,r) = (-r )^N\frac{(\mu_+^{1+N}-\mu_-^{1+N})}{(\mu_+-\mu_-)}= (-r )^N \sum_{j=0}^N \mu_+^j \mu_-^{N-j}.
\end{equation}
If $\mu_+ = \mu_- = \mu$, then using the continuity of the determinant this formula reduces to
$$
\det T_N(s,r) = N (-r \mu)^N.
$$

\subsection{Generalized Euler's identity}
\label{sec:euler identity}
The formula for the determinant of a tridiagonal matrix has been known for a long time. However, we would like to underline that control theoretic approach can be conceptually advantageous. To illustrate this claim, we now prove the generalized Euler's identity~\eqref{eq:euler_agrachev}.

Notice that we get the classical Euler's identity if
$a=0$:
$$
\frac{\sin b}{b} = \prod_{j=1}^\infty\left(1-\frac{b^2}{ (\pi j)^2} \right).
$$
The identity~\eqref{eq:euler_agrachev} appeared in \cite{determinant}. Here we give a proof which uses only elementary tools from linear algebra and analysis. The idea is to take a certain sequence of tridiagonal Toeplitz matrices of increasing rank and to compute the asymptotic expansion of the determinants in two different ways: using formula~\eqref{eq:det_toeplitz_eigenvalues} and~\eqref{eq:det_Toeplitz_control}.

We start with the following family of one-dimensional LQR problems
$$
\dot{x}= ax+u, \quad \frac{1}{2}\int_0^1 \left(u(t)^2 -\varepsilon(a^2+b^2) x^2(t)\right)dt\to \min.
$$
and $x(0)=x(1)=0$. Here $\varepsilon\in \{0,1\}$. Morally we want to compute the determinant of the quadratic functional restricted to the curves which satisfy the constraints. The difficulty here is that the functional in question is a quadratic form on an infinite-dimensional space. We will compute the determinant using finite-dimensional approximations coming from discretizations of this continuous LQR problem with time step equal to $1/N$. We consider discrete LQR problems of the form
$$
x_{k+1}-x_k= \frac{1}{N}(ax_k + u_k), \,\, \frac{1}{2N}\sum_{k=0}^N u_k^2 -\varepsilon (a^2+b^2) x_k^2 \to \min.
$$
In order to apply our formula we introduce a new control function $v= \sqrt{N} u$, which gives an equivalent optimal control problem:
$$
x_{k+1}-x_k= \frac{a}{N}x_k+\frac{v_k}{\sqrt{N}}, \quad \frac{1}{2}\sum_{k=0}^N v_k^2 -\varepsilon\frac{a^2+b^2}{N} x^2 \to \min.
$$
Hence we get as parameters of the LQR problem
\begin{align*}
A_k= aN^{-1}, \quad B_k= N^{-1/2},\\ Q_k = \varepsilon(a^2+b^2)N^{-1}, \quad \Gamma_k = N^{-1}.
\end{align*}
We can now use formulas~\eqref{eq:jacobi_transform} to show that we are actually computing the determinants of a tridiagonal Toeplitz matrices $T_N(r,s)$ from the previous subsection with parameters:
$$
r=-(a+N), \qquad s = N + N\left(1+\frac{a}{N}\right)^2 - \varepsilon\frac{a^2+b^2}{N}.
$$
Let us denote those matrices as $\hat{T}_\varepsilon(a,b,N)$.

As discussed previously we want to compute $\det \hat{T}_\varepsilon(a,b,N)$ in two different ways and then take the limit $N\to \infty$. Foreshadowing a bit the computation, the sequence $\det \hat{T}_\varepsilon(a,b,N)$ does not converge. So instead we compute the limit
$$
\lim_{N \to \infty}\frac{\det \hat{T}_1(a,b,N)}{\det \hat{T}_0(a,b,N)}.
$$
There are two motivations for using exactly this ratio. First of all we can notice from formula~\eqref{eq:det_Toeplitz_control} that if two systems have the same matrices $A_k$ and $B_k$, then the ratio of the corresponding determinants depends only on the associated discrete flows $\Phi$. Since we have started with an approximation to a continuous system, it is natural to expect that the flows will converge and this way we obtain a finite quantity. Another reason is the celebrated Gelfand-Yaglom formula in quantum mechanics \cite{dunne}, which computes the ratio of determinants of a 1D quantum particle with a potential and a 1D free quantum particle as a ratio of determinants of two finite-rank matrices.

Let us first use formula~\eqref{eq:det_Toeplitz_control} to compute the ratio. The two eigenvalues of the flow matrix now depend on $a,b,N,\varepsilon$. To keep formulas clean we omit the dependence on $a,b,N$ in our notations and indicate explicit dependence on $\varepsilon$ through a sub-index or super-index. We also assume that $a>0$, $b>0$. Other sign choices will be a direct consequence of~\eqref{eq:euler_agrachev} itself.

First note that $r$ is independent of $\varepsilon$. If we write $\Delta_\varepsilon = s_\varepsilon^2-r^2$, then the formula~\eqref{eq:eigenvalues_toeplitz} for eigenvalues reads as
$$
\mu^{\varepsilon}_\pm = \frac{-s_\varepsilon \pm \sqrt{\Delta_\varepsilon}}{2r}.
$$
Formula~\eqref{eq:det_Toeplitz_control} in this notations reads as
\begin{align*}
\frac{\det \hat{T}_1(a,b,N)}{\det \hat{T}_0(a,b,N)}=\frac{\left((\mu_+^1)^{1+N}-(\mu_-^1)^{1+N}\right)}{\left((\mu_+^0)^{1+N}-(\mu_-^0)^{1+N}\right)}\frac{(\mu_+^0 - \mu_-^0)}{(\mu_+^1 - \mu_-^1)} \\
=\frac{\sqrt{\Delta_0}}{\sqrt{\Delta_1}}\frac{\left((-s_1 + \sqrt{\Delta_1})^{1+N}-(-s_1 - \sqrt{\Delta_1})^{1+N}\right)}{\left((-s_0 + \sqrt{\Delta_0})^{1+N}-(-s_0 - \sqrt{\Delta_0})^{1+N}\right)}.
\end{align*}

Plugging formulas for $r$ and $s_\varepsilon$ we find that 
$$
\Delta_\varepsilon =  4(1-\varepsilon)a^2 -4 \varepsilon b^2 + O(N^{-1}), \qquad N\to \infty.
$$
Hence
$$
\lim_{N\to \infty }\frac{\sqrt{\Delta_0}}{\sqrt{\Delta_1}} = \frac{2a}{2bi}.
$$
Similarly we find that
$$
s_\varepsilon = 2N + O(1), \qquad N\to \infty.
$$
Hence
$$
\lim_{N\to \infty } \frac{(-s_1\pm\sqrt{\Delta_1})^{1+N}}{(2N)^{1+N}} = e^{\pm ib} 
$$
and
$$
\lim_{N\to \infty } \frac{(-s_0\pm\sqrt{\Delta_0})^{1+N}}{(2N)^{1+N}} = e^{\pm a}. 
$$
Thus after collecting everything together, we find
$$
\lim_{N \to \infty}\frac{\det \hat{T}_1(a,b,N)}{\det \hat{T}_0(a,b,N)} = \frac{2a}{2bi}\frac{e^{ib}-e^{-ib}}{e^a-e^{-a}} = \frac{a \sin b}{b\sinh a } 
$$
which is exactly the left side of~\eqref{eq:euler_agrachev}.

Now we use formula~\eqref{eq:det_toeplitz_eigenvalues} to compute the same ratio. We get
\begin{align*}
&\frac{\prod_{j=1}^N \left( N + N(1+\frac{a}{N})^2+ 2(a+N)\cos \frac{\pi j}{N+1}-\frac{a^2+b^2}{N}\right)}{\prod_{j=1}^N \left( N + N(1+\frac{a}{N})^2+ 2(a+N)\cos \frac{\pi j}{N+1}\right)} \\
&=\frac{\det \hat{T}_1(a,b,N)}{\det \hat{T}_0(a,b,N)}.
\end{align*}

Expanding and simplifying the denominator gives us
\begin{align*}
N + N\left(1+\frac{a}{N}\right)^2+ 2(a+N)\cos \frac{\pi j}{N+1} \\ = \frac{1}{N}\left(a^2+4(N^2+aN)\cos^2 \frac{\pi j}{2(N+1)}\right).
\end{align*}
Hence
$$
\frac{\det \hat{T}_1(a,b,N)}{\det \hat{T}_0(a,b,N)} = \prod_{j=1}^N \left(1- \frac{a^2+b^2}{a^2+4(N^2+aN)\cos^2 \frac{\pi j}{2(N+1)}} \right).
$$
Changing the summation index $j\mapsto N+1 - j$ gives
$$
\frac{\det \hat{T}_1(a,b,N)}{\det \hat{T}_0(a,b,N)} = \prod_{j=1}^N \left(1- \frac{a^2+b^2}{a^2+4(N^2+aN)\sin^2 \frac{\pi j}{2(N+1)}} \right). 
$$

Since $ \sin(y) \le \vert y\vert $, we have the following inequality:
\begin{equation*}
\frac{4(N^2+aN)}{(\pi j)^2}\sin^2 \frac{\pi j}{2(N+1)} \le 1+\frac{(a-2)N-1}{4(N+1)^2} 
\end{equation*}
Thus for any $\epsilon>0$ there exists $N_0\in \mathbb{N}$  such that $ \forall N\ge N_0$:
\begin{equation}
	\label{eq:estimate_tail}
	 \frac{a^2+b^2}{a^2+4(N^2+aN)\sin^2 \frac{\pi j}{2(N+1)}} \ge \frac{a^2+b^2}{a^2 +(1+\epsilon)(\pi j)^2}
\end{equation}
Define the following partial product: 
$$
\Pi_{k}^m(N) = \prod_{j=k}^m \left(1- \frac{a^2+b^2}{a^2+4(N^2+aN)\sin^2 \frac{\pi j}{2(N+1)}} \right).
$$

By continuity we have that:
\begin{align*}
	\lim_{N\mapsto \infty} \Pi_{0}^N(N) = \lim_{N \mapsto \infty} \Pi_0^{N_0}(N) \lim_{N\to\infty} \Pi_{N_0}^N(N)\\
	= \prod_{j=0}^{N_0} \left(1-\frac{a^2+b^2}{a^2+(\pi j)^2}\right) \lim_{N\to\infty} \Pi_{N_0}^N(N).
\end{align*}
Since the second factor is the tail of a convergent product, thanks to the estimates in \cref{eq:estimate_tail}, it converges to $1$ when we take the limit for $N_0\to \infty$ and get:
\begin{equation*}
	\lim_{N_0\to\infty} \prod_{j=0}^{N_0} \left(1-\frac{a^2+b^2}{a^2+(\pi j)^2}\right) = \lim_{N \mapsto \infty} \Pi_0^N(N).
\end{equation*}
 Thus we have proven~\eqref{eq:euler_agrachev}.

\bibliographystyle{plain}
\bibliography{ref}

\end{document}